\documentclass{article}
\usepackage[utf8]{inputenc}
\usepackage[margin=1in]{geometry}          
\usepackage{graphicx}
\usepackage{blkarray}
\usepackage{amsmath, amsthm, amssymb}
\usepackage{setspace}
\usepackage{mathtools}

\let\emptyset\varnothing

\title{A simple proof of the non-uniform Kahn-Kalai conjecture}
\author{Bryan Park \and Jan Vondr\'ak}
\date{\today}

\begin{document}

\maketitle

\newtheorem{theorem}{Theorem}
\newtheorem{problem}{Problem}
\newtheorem{lemma}{Lemma}
\newtheorem{definition}{Definition}
\newtheorem{example}{Example}
\newtheorem{observation}{Observation}
\newtheorem*{answer}{Answer}
\newtheorem{corollary}{Corollary}
\newtheorem{proposition}{Proposition}
\newtheorem*{remark}{Remark}

\newcommand{\Var}{\operatorname{Var}}
\newcommand{\E}{\mathbb{E}}
\newcommand{\N}{\mathbb{N}}
\newcommand{\G}{\mathcal{G}}
\newcommand{\F}{\mathcal{F}}
\newcommand{\FM}{\mathcal{F}^*}
\renewcommand{\H}{\mathcal{H}}
\renewcommand{\S}{\mathcal{S}}
\renewcommand{\L}{\mathcal{L}}
\renewcommand{\P}{\mathbb{P}}
\newcommand{\R}{\mathbb{R}}
\newcommand*{\vv}[1]{\mathbf{\mkern0mu#1}}

\begin{abstract}
We revisit the Kahn-Kalai conjecture, recently proved in striking fashion by Park and Pham \cite{PP}, and present a slightly reformulated simple proof which has a few advantages: (1) it works for non-uniform product measures, (2) it gives near-optimal bounds even for sampling probabilities close to $1$, (3) it gives a clean bound of $p_c \leq 4q_c \log_2 (7\ell)$ for every $\ell$-bounded set system, $\ell\geq 1$. 
\end{abstract}

\section{Introduction}

 Let $X$ be a finite set and $|X|=n.$ Let $\vv{p}\in (0,1)^X$ and let $\mu_{\vv{p}}$ denote the product measure on $2^X$ given by  $\mu_{\vv{p}}(S)=\prod_{x\in S}p_x\prod_{y\notin S}(1-p_y).$
 Let $X_{\vv{p}}$ denote a random set sampled from the distribution $\mu_{\vv{p}}.$ 
 Given a family $\F \subseteq 2^X$ which is \textit{non-trivial} ($\F\neq\emptyset,2^X$) and \textit{up-monotone} (if $A \in \F$ and $A \subseteq B$ then $B \in \F$), we are interested in the probability that $X_{\vv p}\in \F$. The Kahn-Kalai conjecture \cite{KK} concerns the comparison between two notions of a threshold for such an event. Stating the conjecture requires a bit of notation which we adopt from \cite{PP}.

\begin{definition}[$\ell$-bounded]
    \normalfont
    $\G\subseteq 2^X$ is $\ell$-\textit{bounded} if $|S|\leq \ell$ for all $S\in\G.$
\end{definition}

\begin{definition}[cover]
    \normalfont
    Given $\G\subseteq 2^X,$ let $\langle \G\rangle:=\bigcup_{S\in \G}\{T: T\supseteq S\}$. If $\H \subseteq \langle \G \rangle$, we say that $\G$ \textit{covers} $\H$.
\end{definition}

\begin{definition}[$\vv q$-cost, $\vv q$-small]
\normalfont 
    Given $\H\subseteq 2^X$ and $\vv{q} \in (0,1)^X$, let $e_{\vv{q}}(\H) := \sum_{S \in \H} \prod_{x \in S} q_x$. Equivalently, $e_{\vv q}(\H)=\E[|\{S\in \H: X_{\vv q}\supseteq S\}|]$. Let $c_{\vv{q}}(\H) := \min\{e_{\vv{q}}(\G)\mid \H\subseteq \langle \G\rangle\}$.
    We say that $c_{\vv q}(\H)$ is the \textit{$\vv q$-cost} of $\H$ and $\H$ is \textit{$\vv q$-small} if $c_{\vv q}(\H) \leq 1/2$.
\end{definition}

In the uniform setting, we have a single probability $p \in (0,1)$ instead of $\vv{p} \in (0,1)^X$ and all the notions above are considered with the uniform probability vector $\vv{p} = (p,p,\ldots,p)$. 
Given a non-trivial, up-monotone family $\F \subseteq 2^X$, we can compare two thresholds: The \textit{probability threshold} $p_c := \sup \{p: \Pr[X_p \in \F] \leq 1/2 \}$, and the \textit{expectation threshold} $q_c := \sup \{q: \F \mbox{ is }   q\mbox{-small}\}$. It is easy to verify that $q_c \leq p_c$. Kahn and Kalai \cite{KK} conjectured that the ratio between the two thresholds is always $O(\log n) = O(\log |X|)$, which implies several difficult results in probabilistic combinatorics. This conjecture remained open until last year when Park and Pham found a shockingly elegant proof \cite{PP}.

In this note, we give a reformulation and simplification of their proof, which yields the following.\footnote{All logarithms in this paper are base 2.}

\begin{theorem}
    Let $L=4,$ $\ell\geq 1,$ $\vv q \in (0,1)^X$, and $p_x=1-(1-q_x)^{L\lfloor \log(2\ell) \rfloor+7}$ for $x\in X$. For any $\ell$-bounded $\H\subseteq 2^X$, if $c_{\vv q}(\H) > 1/2,$ then $\Pr[X_{\vv{p} }\in\langle\H\rangle] > 1/2$.
\end{theorem}

A natural interpretation of Theorem 1 is that if $\H$ is not $\vv{q}$-small, then a union of $O(\log \ell)$ independent copies of $X_{\vv{q}}$ is in $\langle \H \rangle$ with constant probability.
In particular, since $p_x \leq q_x(4\log(2\ell)+7)$ for $x\in X$, Theorem 1 implies that $p_c \leq q_c (4\log (2 \ell)+7)\leq 4q_c\log(7\ell)$, i.e.~the uniform Kahn-Kalai conjecture. 
However, Theorem 1 also gives a meaningful result when the $q_x$'s are non-uniform, and possibly close to $1$. 

We remark that the best known multiplicative constant is $L \approx 3.995$ \cite{Huy,VT}. However this comes with a rather large additive term, and we do not pursue this bound here. (It could be obtained with some additional work, as our method is fundamentally the same as \cite{PP,VT}.) Alternatively, our method also gives the bound $p_c \leq 4.5 q_c \log (2 \ell)$ which strictly improves $p_c \leq 8 q_c \log (2 \ell)$ given in \cite{Bell}. 
 
Our proof has the same structure as the proof of Park and Pham \cite{PP}, and the same key concept of \textit{minimal fragments}. The main difference is that we work directly with random sets under a product measure rather than random sets of a fixed size. As a result, we are able to streamline the argument of \cite{PP} and apply it to non-uniform product measures. We also remark that another way to prove a non-uniform version of the Kahn-Kalai conjecture appears to be an element duplication scheme along the lines of \cite{PR}. However, this would incur additional technicalities and a weaker quantitative bound.

\section{The Proof}

We begin with some preliminaries (adopted from \cite{PP}).

\begin{definition}[fragments and minimal fragments]
    \normalfont
    Given $\H\subseteq 2^X$  and $W\in 2^X,$ any element of $\F(\H,W):=\{S\setminus W\mid S\in \H\}$ is called a \textit{fragment} of $\H$ and $W.$ We denote by $\FM(\H,W)$  the set of (inclusion-wise) minimal elements of $\F(\H,W)$, which we call \textit{minimal fragments}.
\end{definition}

\begin{definition}[large and small minimal fragments]
\normalfont
Given $\H\subseteq 2^X$ and  $W\in 2^X$,  we define
$\L_m(\H,W) := \{ T \in \FM(\H,W): |T| \geq m \}$, the set of ``\textit{large}" minimal fragments. Similarly, we define $\S_m(\H,W):=\FM(\H,W)\setminus \L_m(\H,W) = \{ T \in \FM(\H,W): |T| < m \}$, the set of ``\textit{small}" minimal fragments.
\end{definition}

Moreover, note that ${\vv q}$-cost is subadditive: If $\H=\H_1 \cup \H_2,$ then $c_{\vv q}(\H)\leq c_{\vv q}(\H_1)+c_{\vv q}(\H_2)$, since a union of the respective  covers of $\H_1$ and $\H_2$ is a cover of $\H$.

\subsection{The $\vv{q}$-cost Lemma}
Our proof of Theorem 1 involves the following $\vv{q}$-cost lemma analogous to Lemma 2.1 in \cite{PP}. 

\begin{lemma}
Let $L\geq 1$, $\ell \geq m \geq 1$, $\vv{q}\in (0,1)^X,$ and $p_x = 1 - (1-q_x)^L$ for $x\in X$. If $\H\subseteq 2^X$ is $\ell$-bounded, then 
$$\E[c_{\vv q}(\L_m(\H,X_{\vv p}))] \leq \Pr[X_{\vv p} \in \langle \H \rangle] \cdot \sum_{j=m}^{\ell} \frac{1}{L^j} \binom{\ell}{j}.$$
\end{lemma}

We remark that a version analogous to \cite{PP} would have a slightly looser bound of $2^\ell / L^m$ on the right-hand side. The difference between the two leads to different constants in the main theorem.

\begin{proof}
  In the following, we use the notation $\vv{q}^S:=\prod_{x \in S} q_x$ and $(\vv{1} - \vv{q})^T = \prod_{x \in T} (1-q_x)$. 
  By definition,
    \begin{align*}
    \E[c_{\vv q}(\L_m(\H,X_{\vv p}))] \leq \E[e_{\vv q}(\L_m(\H,X_{\vv p})]& = \sum_{W\in 2^X} \vv{p}^{W} (\vv{1}-\vv{p})^{X \setminus W}
        \sum_{S\in \L_m(\H,W)} \vv{q}^{S}.
    \end{align*}
    The following is the key double-counting trick of Park and Pham:
    Let $Z = W \cup S$ (a disjoint union) and rewrite the summation over $W$ and $S$ as a summation over $Z$ and $S$, where $W = Z \setminus S$:
    \begin{align*}
        \E[c_{\vv q}(\L_m(\H,X_{\vv p})]
        &\leq \sum_{Z \in 2^X} \sum_{\substack{S\subseteq Z,\\S\in \L_m(\H,Z\setminus S)}} {\vv p}^{Z \setminus S}({\vv 1}-{\vv p})^{X \setminus (Z \setminus S)} {\vv q}^{S} \\
        &= \sum_{Z \in 2^X}  {\vv p}^{Z} ({\vv 1}-{\vv p})^{X \setminus Z} \sum_{\substack{S\subseteq Z,\\S\in \L_m(\H,Z\setminus S)}} \frac{{\vv q}^S ({\vv 1}-{\vv p})^{S}}{{\vv p}^S}.
    \end{align*}
    If $Z$ contains no set in $\H$, then there is no $S\subseteq Z$ such that $S \in \L_m(\H,Z \setminus S)$ and thus no contribution to the inner sum. Otherwise, fix some $T_Z\in \H$ such that $T_Z\subseteq Z$. Then, for any $S\subseteq Z$ we have by definition 
    $S \cap T_Z \in \F(\H,Z \setminus S)$. This means that for any $S\subseteq Z$, if $S \in \F^*(\H, Z \setminus S)$, then $S\subseteq T_Z$. In other words, $T_Z$ contains {\em all} the minimal fragments in $\FM(\H, Z \setminus S)$, and hence also all the sets in $\L_m(\H,W)$. Therefore, we have
    \begin{align*}
        \E[c_{\vv q}(\L_m(\H,X_{\vv p})]
        &\leq \sum_{\substack{Z \in 2^X, \\ \exists T_Z \subseteq Z, T_Z \in \H}} {\vv p}^{Z} ({\vv 1}-{\vv p})^{X \setminus Z} \sum_{\substack{S\subseteq T_Z, |S| \geq m}} \frac{{\vv q}^S (1-{\vv p})^{S}}{{\vv p}^S}.
    \end{align*}
    Note that the expression at the end is $$\frac{{\vv q}^S ({\vv 1}-{\vv p})^S}{{\vv p}^S} = \prod_{x \in S} \frac{q_x (1-p_x)}{p_x} = \prod_{x \in S} \frac{q_x (1-q_x)^L}{1 - (1-q_x)^L}.$$
    The function $\frac{q_x (1-q_x)^L}{1-(1-q_x)^L}$ is decreasing for $q_x \in (0,1)$, as can be verified by taking the derivative, and the limiting value as $q_x \to 0$ is $1/L$. Hence, we have
    $\frac{{\vv q}^S ({\vv 1}-{\vv p})^S}{{\vv p}^S} \leq \frac{1}{L^{|S|}}.$
    Summing up over all $S \subseteq T_Z, |S| \geq m$, we obtain
     $\sum_{j=m}^{\ell} \frac{1}{L^j} \binom{\ell}{j}$ since $\H$ is $\ell$-bounded and $T_Z\in\H$. Finally, $\sum_{Z\in 2^X, \exists T_Z \subseteq Z, T_Z \in \H} {\vv p}^{Z} ({\vv 1}-{\vv p})^{X \setminus Z}$ is exactly the probability that $X_{\vv{p}} \in \langle \H \rangle$ as desired.
\end{proof}

In the special case $\ell=m=1,$ we can prove directly the following bound. 
\begin{lemma}
    Let $L\geq 1$, $\ell =m= 1$, $\vv{q}\in (0,1)^X,$ and $p_i = 1 - (1-q_x)^L$ for $x\in X$. If $\H\subseteq 2^X$ is $\ell$-bounded,
$$\E[c_{\vv q}(\L_m(\H,X_{\vv p}))] \leq \frac{1}{eL}.$$
\end{lemma}
\begin{proof}
    Since $\H$ is $1$-bounded, we can assume that $\H = \{ \{h_1\},\ldots,\{h_r\} \}$. If $X_{\vv p}$ contains any of the elements $h_i$, then $\FM(\H,X_{\vv p})$ contains the empty set, which is the only minimal fragment and hence $\L_1(\H,X_{\vv p})$ is empty. Otherwise, we see that $\L_1(\H,X_{\vv p})=\H.$ Therefore, we have
    \begin{align*}
        \E[c_{\vv q}(\L_1(\H,X_{\vv p}))] &= c_{\vv q}(\H) \cdot \Pr[X_{\vv p} \cap \{h_1,\ldots,h_r\} = \emptyset]
        \leq \left(\sum_{i=1}^r q_{h_i}\right)\cdot \prod_{i=1}^r (1-q_{h_i})^{L}.
    \end{align*}
    Since $1-t\leq e^{-t}$ for $t\in \R$ and $\frac{1}{eL}$ is the maximum of the function $x e^{-Lx}$ for $x \geq 0$, we conclude that
    \begin{align*}
        \E[c_{\vv q}(\L_1(\H,X_{\vv p}))]
        &\leq  \left(\sum_{i=1}^r q_{h_i}\right)\cdot e^{-L\sum_{i=1}^r q_{h_i}}\leq \frac{1}{eL}.
    \end{align*}
\end{proof}

\subsection{Proof of Theorem 1}

\begin{proof}[Proof of Theorem 1]
    Let $k=\lfloor{\log (2\ell)}\rfloor\geq 1$ such that $\ell\in [2^{k-1},2^k).$ Let $\{L_i\}_{i\geq 1}$ be a sequence of reals such that $L_i\geq 1$ for all $i\geq 1.$ Let $r_x^{(i)}=1-(1-q_x)^{L_i}$ for each $i\geq 1$ and $x\in X.$ Finally, let $\H_0:=\H$ and inductively define $\H_i=\S_{2^{k-i}}(\H_{i-1},W_i)$ for $i\in [k]$ where $W_i\sim \mu_{\mathbf{r}^{(k+1-i)}}$ are independent. Let $W:=\bigcup_{i=1}^k W_i.$ Note that $W\sim \mu_{\mathbf{p}}$ where $p_x=1-(1-q_x)^{L_1+\cdots + L_k}$ for each $x\in X$. We aim to choose $\{L_i\}_{i\geq 1}$ that minimizes $\sum_{i=1}^k L_i$ and satisfies $\Pr[W\in\langle\H\rangle]>1/2.$
    
    By construction, $\H_i$ is $(2^{k-i}-1)$-bounded for $i\in \{0,1,\dots,k\}$. In particular, $\H_k$ is $0$-bounded. Thus, $\H_{k}=\emptyset$ or $\H_{k}=\{\emptyset\}$. If the event $E=\{\H_{k}=\{\emptyset\}\}$ occurs, this means that some fragment of a set in $\H$ eventually became $\emptyset$ after iterating $k$ times, which means that $W$ contains this set.  Depending on whether $E$ holds, the $\vv q$-cost at the end is $c_{\vv q}(\H_{k}) = 1$ or $0$.  Hence, if $c_{\vv q}(\H_k)>0$ then $E$ occurs and thus $W\in\langle \H\rangle$. 
    
    By subadditivity of $c_{\vv q},$ for each $i\in [k]$ we have 
    \begin{align*}
       c_{\vv q}(\H_i)&\geq c_{\vv q}(\FM(\H_{i-1},W_i))-c_{\vv q}(\L_{2^{k-i}}(\H_{i-1},W_i))\geq c_{\vv q}(\H_{i-1})-c_{\vv q}(\L_{2^{k-i}}(\H_{i-1},W_i)).
    \end{align*}
    To check the second inequality, note that if $\FM(\H_{i-1},W_i)\subseteq \langle \G\rangle,$ then $\H_{i-1}\subseteq \langle \G\rangle$ since $\FM(\H_{i-1},W_i)$ contains a fragment of every set in $\H_{i-1}$. Iterating through $i\in [k],$ we have 
    \begin{align*}
        c_{\vv q}(\H_k)\geq c_{\vv q}(\H)-\sum_{i=1}^k c_{\vv q}(\L_{2^{k-i}}(\H_{i-1},W_i)).
    \end{align*}
    Let $Z:=\sum_{i=1}^k c_{\vv q}(\L_{2^{k-i}}(\H_{i-1},W_i))$. 
    If $Z<c_{\vv q}(\H),$ then we have $c_{\vv q}(\H_k)>0$ and thus $W\in\langle \H\rangle.$ Hence, we will choose $\{L_i\}_{i\geq 1}$ such that $\P(Z<c_{\vv q}(\H))>1/2$ while minimizing $\sum_{i=1}^k L_i.$ 
    
    Let $m_i=2^{k-i}$ for $i\in [k].$ Applying Lemma 1 for $i \in [k-1]$ with parameters $\ell=2m_i-1 = 2^{k-i+1}-1$ and $m=m_i=2^{k-i}$, we have
    \begin{align*}
        \E[c_{\vv q}(\L_{m_i}(\H_{i-1},W_i))\mid \H_{i-1}] &\leq \Pr[W_i\in \langle \H_{i-1}\rangle\mid \H_{i-1}]\sum_{j=m_i}^{2m_i-1} \frac{1}{(L_{k+1-i})^j}\binom{2m_i-1}{j}.
    \end{align*}
    Taking expectations over $\H_{i-1}$ for both sides and noting $$\Pr[W_i\in\langle \H_{i-1}\rangle] \leq \Pr[W_1\cup\cdots \cup W_i\in \langle\H\rangle]\leq \Pr[W\in\langle \H\rangle],$$ we have
    \begin{align*}
        \E[c_{\vv q}(\L_{2^{k-i}}(\H_{i-1},W_i))]
        &\leq
        \Pr[W\in\langle \H\rangle]\sum_{j=m_i}^{2m_i-1} \frac{1}{(L_{k+1-i})^j}\binom{2m_i-1}{j}.
    \end{align*}
    If $i=k,$ since $\H_{k-1}$ is $1$-bounded, by Lemma 2 we have $\E[c_{\vv q}(\L_1(\H_{k-1},W_k))]\leq 1/(eL_1).$
    Therefore, we have
    \begin{align*}
        \E[Z]&\leq \frac{1}{eL_1} + \sum_{i=1}^{k-1}\Pr[W\in\langle \H\rangle]\sum_{j=m_i}^{2m_i-1} \frac{1}{(L_{k+1-i})^j}\binom{2m_i-1}{j}\\
        &\leq \frac{1}{eL_1} +\Pr[W\in\langle \H\rangle]\sum_{i=2}^{\infty}\sum_{j=2^{i-1}}^{2^i-1}\frac{1}{(L_i)^j}\binom{2^i-1}{j}.
    \end{align*}
    
    By Markov's inequality, $\Pr[Z \geq c_\vv{q}(\H)] \leq \frac{\E[Z]}{c_\vv{q}(\H)}$. If $c_{\vv q}(\H)\leq \P[W\in\langle \H\rangle],$ then $\P[W\in\langle \H\rangle]>1/2$ by the hypothesis. Otherwise, we use $c_{\vv q}(\H) > 1/2$ and $c_{\vv q}(\H)>\P[W\in\langle \H\rangle]$ to deduce 
    \begin{align}
        \Pr[Z\geq c_{\vv q}(\H)] 
        &\leq\frac{2}{eL_1} + \sum_{i=2}^{\infty}\sum_{j=2^{i-1}}^{2^i-1}\frac{1}{(L_i)^j}\binom{2^i-1}{j}.\label{eq:key}
    \end{align}
    Hence, if we choose $\{L_i\}_{i\geq 1}$ such that the right-hand side of \eqref{eq:key} is less than $1/2$, then we always have $\Pr[W\in\langle \H\rangle]>1/2$ as desired. In particular, it is easy to see that $L_i = L$ for $L > 4$ large enough makes \eqref{eq:key} less than $1/2$, which is sufficient to prove the Kahn-Kalai conjecture.

    In order to obtain a good quantitative bound, we want to find a choice of $\{L_i\}_{i\geq 1}$ that also minimizes $\sum_{i=1}^k L_i.$ 
    The bound of Theorem 1 is obtained by $L_i=5$ if $i\leq 7$ and $L_i=4$ if $i>7.$ This gives $W\sim \mu_{\vv p}$ where $p_x\leq 1-(1-q_x)^{4\lfloor\log(2\ell)\rfloor + 7}$ which is our desired random set. We now show that 
    \begin{align*}
        \frac{2}{5e} + \sum_{i=2}^{7}\sum_{j=2^{i-1}}^{2^i-1}\frac{1}{5^j}\binom{2^i-1}{j} + \sum_{i=8}^{\infty}\sum_{j=2^{i-1}}^{2^i-1}\frac{1}{4^j}\binom{2^i-1}{j}<\frac{1}{2}.
    \end{align*}
    To see this, first note that 
    \begin{align*}
        \sum_{j=2^{i-1}}^{2^i-1}\frac{1}{4^j}\binom{2^i-1}{j} &\leq \binom{2^i-1}{2^{i-1}} \sum_{j=2^{i-1}}^{\infty} \frac{1}{4^j} = \binom{2^i-1}{2^{i-1}}\frac{1}{4^{2^{i-1}}(1-1/4)}.
    \end{align*}
    Using $\binom{2n-1}{n}=\binom{2n}{n}/2 < 2^{2n-1}/\sqrt{\pi n},$ we obtain
    \begin{align*}
        \sum_{i=8}^{\infty}\sum_{j=2^{i-1}}^{2^i-1}\frac{1}{4^j}\binom{2^i-1}{j} &\leq \sum_{i=8}^\infty \frac{2^{2^i-1}}{\sqrt{\pi 2^{i-1}}4^{2^{i-1}}(1-1/4)}
        = \sum_{i=8}^\infty \frac{2}{3\sqrt{\pi 2^{i-1}}}=\frac{1+\sqrt{2}}{12\sqrt{\pi}}.
    \end{align*}
    Bringing everything together, we have
    \begin{align*}
        \frac{2}{5e} + \sum_{i=2}^{7}\sum_{j=2^{i-1}}^{2^i-1}\frac{1}{5^j}\binom{2^i-1}{j} + \frac{1+\sqrt{2}}{12\sqrt{\pi}} < \frac{1}{2}
    \end{align*}
    which can be verified by computer. This concludes our proof.
\end{proof}

\paragraph{Remark.}
The additive constant $7$ in $4\lfloor \log(2\ell)\rfloor + 7$ could be optimized slightly further, since one can use constants smaller than $5$ to define $L_i$ when $i\leq 7$.

\subsection{Concluding Remarks}

Alternative bounds can be obtained by other choices of $\{L_i\}_{i\geq 1}$ in the proof of Theorem 1. First, assume that $L_i=L$ for all $i\geq 1$ and recall $\eqref{eq:key}$ from the proof of Theorem 1. As we mentioned, it is easy to see that as $L \to \infty$, the right-hand side converges to $0$, which is 
sufficient to prove the Kahn-Kalai conjecture. 

In fact, setting $L=4.5$ works, which can be verified (by computer) similar to the final calculations in our proof of Theorem 1. Indeed, if $L=4.5,$ then $W\sim \mu_{\vv p}$ where $p_x=1-(1-q_x)^{4.5\lfloor\log(2\ell)\rfloor}$, which implies the result $p_c\leq 4.5q_c\log(2\ell)$ mentioned in the introduction.

To obtain a concrete bound which can be verified easily by hand, we can set $L=6$. Again recall \eqref{eq:key} in the proof of Theorem 1. Using Lemma 1 for \textit{all} $i\in [k]$ instead of just $i\in [k-1]$ and letting $L_i=L=6$ for all $i\geq 1,$ we have (if $c_{\vv q}(\H)>\P[W\in\langle \H\rangle]$) the inequality
\begin{align*}
        \Pr[Z\geq c_{\vv p}(\H)]&\leq \sum_{i=1}^{\infty}\sum_{j=2^{i-1}}^{2^i-1}\frac{1}{L^j}\binom{2^i-1}{j}.
    \end{align*}
Note that the term $L^{-j}$ appears exactly once in the summation above for each $j\geq 1.$ For $j\in [4],$ take the exact coefficients of $L^{-j}.$ For $j\geq 5,$ we use $\binom{2^i-1}{j}\leq \binom{2j-1}{j}\leq 2^{2j-1}$ to get
\begin{align*}
        \Pr[Z\geq c_{\vv p}(\H)]
        &\leq \frac{1}{L} + \frac{3}{L^2} + \frac{1}{L^3} + \frac{35}{L^4} + \sum_{j=5}^\infty \frac{2^{2j-1}}{L^j}\\
        &= \frac{1}{L} + \frac{3}{L^2} + \frac{1}{L^3} + \frac{35}{L^4} + \frac{(4/L)^5}{2(1-4/L)} = \frac{23}{48} < \frac{1}{2}
\end{align*}
as desired. This gives $p_c\leq 6q_c\log(2\ell)$.

As long as we define the sequence of hypergraphs $\H_i$ to be $(2^{k-i}-1)$-bounded, the proof does not go through for a multiplicative constant $L$ smaller than $4$. However, it is possible to achieve a multiplicative constant slightly less than $4$ with a factor slightly different from $2$ to define the iterative process. 
 E.g. \cite{VT} achieves a constant $L \simeq 3.998$ for the Kahn-Kalai conjecture and \cite{Huy} cites an improvement to $L \simeq 3.995$. These improvements come with a larger additive term, as the summation analogous to \eqref{eq:key} does not converge very quickly for $L < 4$. We do not pursue this here. We note that the bound proved in \cite{VT} is roughly $p_c \leq ((3.998 + \delta) \log \ell + 1000 \log (1/\delta)) q_c$, which beats our bound of $p_c \leq 4 q_c \log (7 \ell)$ asymptotically, but our bound is stronger for $\ell < 2^{1,000,000}$.
It appears that a new idea would be required to push the multiplicative constant significantly below $4$.


\begin{thebibliography}{5}
\bibitem{Bell} T. Bell. The Park-Pham theorem with optimal convergence rate. \textit{Electronic Journal of Combinatorics}, 30(2), 2--25, 2023.
\bibitem{KK} J. Kahn and G. Kalai. Thresholds and expectation thresholds. \textit{ Combin. Probab. Comput.} 16, 495–-502,  2007.
\bibitem{PP} J. Park, H. T. Pham. A proof of the Kahn-Kalai conjecture. Preprint, arXiv:2203.17207v2, 2022.
\bibitem{Huy} H. T. Pham. Personal communication, 2023.
\bibitem{PR} Tomasz Przybylowski and O. Riordan. Thresholds and expectation thresholds for larger $p$. Preprint, arXiv:2302.03327, 2023.
\bibitem{VT} V. Vu, P. Tran. A short proof of the Kahn-Kalai conjecture. Preprint, arXiv:2303.02144v1, 2023.
\end{thebibliography}
\end{document}